\documentclass[12pt]{article}
\usepackage{latexsym}
\usepackage{amsmath}
\usepackage{amssymb}
\usepackage{amsthm}
\usepackage{mathtools}
\usepackage{dsfont}
\usepackage{bbm}
\usepackage{color}
\usepackage[tmargin=1.25in,bmargin=1.25in,lmargin=1.1in,rmargin=1.1in]{geometry}

\usepackage{graphicx}
\usepackage{float}
\usepackage{subcaption}
\usepackage{epsfig}
\usepackage{hyperref} 

\definecolor{Red}{rgb}{1,0,0}
\definecolor{halfgray}{gray}{0.55}
\definecolor{webgreen}{rgb}{0,.5,0}
\definecolor{webbrown}{rgb}{.6,0,0}
\definecolor{Maroon}{cmyk}{0, 0.87, 0.68, 0.32}
\definecolor{RoyalBlue}{cmyk}{1, 0.50, 0, 0}
\definecolor{Black}{cmyk}{0, 0, 0, 0}

\hypersetup{
    colorlinks=true, linktocpage=true, pdfstartpage=1, pdfstartview=FitV,
    breaklinks=true, pdfpagemode=UseNone, pageanchor=true, pdfpagemode=UseOutlines,
    plainpages=false, bookmarksnumbered, bookmarksopen=true, bookmarksopenlevel=1,
    hypertexnames=true, pdfhighlight=/O,
    urlcolor=webbrown, linkcolor=RoyalBlue, citecolor=webgreen,
}

\newtheorem{thm}{Theorem}[section]
\newtheorem{proposition}[thm]{Proposition}
\newtheorem{cor}[thm]{Corollary}
\newtheorem{lem}[thm]{Lemma}

\makeatletter\@addtoreset{equation}{section}\makeatother

\begin{document}

\newcommand{\D}{\mathrm{d}}
\newcommand{\R}{\mathbb{R}}
\newcommand{\EE}{\mathbb{E}}
\newcommand{\PP}{\mathbb{P}}
\newcommand{\NN}{\mathbb{N}}
\newcommand{\ZZ}{\mathbb{Z}}

\newcommand{\var}{\operatorname{Var}}
\newcommand{\cov}{\operatorname{Cov}}
\newcommand{\G}{\mathcal{G}}
\newcommand{\B}{\mathcal{B}}
\newcommand{\e}{\mathrm{e}}
\newcommand{\p}{\mathrm{P}}
\newcommand{\N}{\mathbb{N}}
\newcommand{\K}{\mathbb{K}}
\newcommand{\T}{\mathbb{T}}
\newcommand{\bbP}{\mathbb{P}}
\newcommand{\bbE}{\mathbb{E}}
\newcommand{\bfP}{\mathbf{P}}
\newcommand{\bfE}{\mathbf{E}}
\newcommand{\calE}{\mathcal E}
\newcommand{\GW}{\operatorname{GW}}
\newcommand{\Tra}{\operatorname{Tra}}
\newcommand{\RW}{\operatorname{RW}}
\newcommand{\RWP}{\RW\times\bfP}
\newcommand{\F}{\mathcal F}
\newcommand{\bo}{\Bar{\omega}}
\newcommand{\Var}{\mathrm{Var}}

\pagenumbering{arabic}

\newcommand{\Peaug}{\mathrm{P}_\varepsilon^{\mathrm{aug}}}
\newcommand{\Pehat}{\hat{\mathrm{P}}_\varepsilon}
\newcommand{\PPehat}{\hat{\mathbb{P}}_\varepsilon}
\newcommand{\Eehat}{\hat{\mathrm{E}}_\varepsilon}
\newcommand{\EEehat}{\hat{\mathbb{E}}_\varepsilon}

\title{Random walk in a regenerating random environment}
\author{Raphael Martin, Jan Nagel} 
\maketitle

\begin{abstract} 
In this paper we consider random walks moving in a dynamic random environment that is completely resampled  with probability $p$ after each step of the random walker. This resampling yields a straightforward regeneration structure and limit theorems, such as a law of large numbers with a limiting speed $v(p)$ that depends on $p$. We investigate several properties of $v(p)$ as a function of $p$, such as differentiability and monotonicity. Similarily, we study the limiting diffusivity $\sigma^2(p)$ in a central limit theorem.  
\end{abstract}


\smallskip


\section{Introduction}
\label{sec:intro}

Random walks in dynamic random environments are an important variant of random walks in static random environments. Instead of sampling the transition probabilities of a random walk $(X_n)_{n\geq 0}$ once and then keeping this environment fixed as the process moves, the transition probabilities are subject to a random dynamic and the random walk $(X_n)_{n\geq 0}$ is then a time-inhomogeneous Markov chain. \cite{boldrighini1997almost} introduced an environment with i.i.d. updates in time and proved a central limit theorem for the position of the walker, which was extended later by \cite{boldrighini2004random,bandyopadhyay2005random}. Limit theorems in more general environment dynamics have been considered for example by \cite{avena2011law,Redig2013,Blondel2020Random}. A random walk in a dynamical percolation model was studied and compared with a static model in \cite{HermonSousi2020,peres2020mixing} and in \cite{andres2024biased}, the influence of a bias parameter on the random walk on dynamical percolation was investigated. Similar to the last reference, we want to understand how limit theorems depend on the model parameters. In particular, how does the rate of the environment dynamics influence the long term behavior of the random walk?    

Such questions require a very precise control of the path behavior, often with some uniformity in the rate of the environment dynamic. We therefore study the following variant of a dynamic environment: After each step of the random walk $(X_n)_{n\geq 0}$ we resample with probability $p$ the whole environment, that is, the transition probabilities at each site,  independently of each other and of the past. This is a variant of the random walk in cooling random environments, as introduced by \cite{avena2019random} and studied in \cite{avena2019randomergodic,avena2022random,da2024limiting,avena2025gaussian}. The main feature of cooling environments is that the whole environment is regenerated at a sequence of deterministic times, such that the time between regenerations diverges, whereas in our model these times are random and geometrically distributed.

Regenerating the whole environment gives a natural decomposition of the random walk path into independent increments, and thus an effective way to analyze asymptotic quantities, such as the effective speed of the random walk. In particular, we are able to state results for the speed as a function of the renewal parameter $p$. Let us emphasize that when $p=0$, the random walk moves in a static random environment. When $p=1$, the environment is resampled after each step, which results in a space-homogeneous random walk. Therefore, the behavior for $p$ close to zero or close to 1 are particular interesting, as they allow for a comparison between dynamic and static environments.

Before we present our results, we first state the precise definition of the model. Later, we will focus on a one-dimensional model in $\mathbb Z$, but there is no difference in introducing the process in $\mathbb Z^d$. 
Let $\omega = (\omega^{(k)})_{k\geq 0}$ be a sequence of environments $\omega^{(k)} = (\omega^{(k)}_x)_{x\in \mathbb Z^d}$, such that for all $k\geq 0$ and $x\in \mathbb Z^d$, $\omega_x^{(k)} = (\omega_x^{(k)}(y))_{y\in \mathbb Z^d}$ is a probability vector, that is, $\omega_x^{(k)}(y)\geq 0$ for all $y\in \mathbb Z^d$ and $\sum_{y}\omega_x^{(k)}(y)=1$. 
We denote the space of all such environments by $\Omega$. Given a strictly increasing integer sequence $\tau=(\tau_k)_{k\geq 0}$ of regeneration times with $\tau_0=0$, we define the time-inhomogeneous Markov chain $(X_n)_{n\geq 0}$ with $X_0=0$ and transition probabilities
\begin{align} \label{eq:transitionprob}
P_{\omega,\tau}(X_{n+1} = x+y | X_n=x) = \omega^{(k)}_x(y) ,
\end{align}  
when $\tau_{k}\leq n<\tau_{k+1}$. The probability measure $P_{\omega,\tau}$ can be viewed as a measure on the path space $\mathbb Z^d$ and is called the \emph{quenched law} of the random walk $(X_n)_{n\geq 0}$. Let $\mathrm{P}_p$ be a measure on the space of all environments $\omega$ and all regeneration sequences $\tau$, such that under $\mathrm{P}_p$, the sequence $\omega = (\omega^{(k)})_{k\geq 0}$ is i.i.d. and for all $k\geq 0$, $\omega^{(k)} = (\omega^{(k)}_x)_{x\in \mathbb Z^d}$ is an i.i.d. family, and $(\tau_{k+1}-\tau_k)_{k\geq 0}$ is an i.i.d. sequence independent of $\omega$, such that $\tau_1-\tau_0$ is geometrically distributed with parameter $p\in (0,1]$. 

The \emph{annealed law} $\PP_p$ of the random walk is then defined by 
\begin{align} \label{defannealedlaw}
\PP_p(A\times B\times C) 
= \int_{A\times B} P_{\omega,\tau}((X_n)_{n\geq 0}\in C) \, d\mathrm{P}_p(\omega,\tau) ,
\end{align}
for measurable sets $A\subset \Omega, B\subset  {\mathbb N_0}^{\mathbb N_0}, C\subset (\mathbb Z^d)^{\mathbb N_0}$ (as usual, the $\sigma$-algebras on those spaces are generated by finite-dimensional projections). Under $\PP_p$, $(X_n)_{n\geq 0}$ is a \emph{random walk in a dynamic random environment}.  

As described above, the regeneration of the environment at times $\tau_k$ is global, that is, all transition probabilities $\omega_x^{(k-1)}$ are replaced by the independent copies $\omega_x^{(k)}$. Therefore, the increments $X_{\tau_k}-X_{\tau_{k-1}}$, $k\geq 1$ are i.i.d. under $\PP_p$ and $\tau_k-\tau_{k-1}$, $k\geq 1$ are i.i.d. as well. 

Since we are interested in the precise change of limit theorems when we vary $p$, we focus in the following on the more accessible one-dimensional case. 
Let us assume that $\omega_x^{(k)}(y)=0$ if $|y|\neq 1$, that is, the random walk is nearest-neighbour. 
In the one-dimensional nearest-neighbour case we can simplify the notation and write $\omega^{(k)}_x$ for $\omega_x^{(k)}(1)$ and especially $\omega_x$ for $\omega_x^{(0)}(1)$.

The strong law of large numbers implies then that the limit 
\begin{align}
v(p) \coloneqq \lim_{n\to \infty} \frac{X_n}{n} 
\end{align}
exists $\PP_p$-almost surely. Indeed, due to the integrability of $\tau_k-\tau_{k-1}$ this limit may be taken along the sequence $\tau_k$, such that  
\begin{align}\label{speed1}
v(p) = \lim_{n\to \infty} \frac{X_{\tau_n}}{\tau_n}  = \lim_{n\to \infty} \frac{\tfrac{1}{n} \sum_{k=1}^n X_{\tau_k}-X_{\tau_{k-1}}}{\tfrac{1}{n} \sum_{k=1}^n {\tau_k}-{\tau_{k-1}}} = \frac{\EE_p[X_{\tau_1}]}{\EE_p[\tau_1]} = p \EE_p[X_{\tau_1}] ,  
\end{align}
$\PP_p$-almost surely, where the last identity follows from the geometric distribution of $\tau_1$. 

When $p=0$, we do not have a well-defined sequence of regeneration times, but we may still consider a random walk in an environment that is resampled with probability $p=0$. This is just a random walk is a static random environment, which is consistent with \eqref{eq:transitionprob}, setting $\tau_0=0$ and $\tau_k=\infty$ for $k\geq 1$. 
 
The corresponding annealed law $\PP_0$ is then defined analogously to \eqref{defannealedlaw}. The law of large numbers for random walk in a static random environment yields the $\PP_0$-almost sure limit
\begin{align*}
v(0) = \lim_{n\to \infty} \frac{X_n}{n} . 
\end{align*}
Let us emphasize that $v(0)$ is the classical speed of a random walk in a static random environment, see \cite{solomon1975random,zeitouni2004RWRE}. For later reference, it will be useful to consider, for any $p\in [0,1]$, a random walk $(Z_n)_{n\geq 0}$ in a static environment. 
Its quenched law may be defined by   
\begin{align}
P_{\omega,\tau}(Z_{n+1} = x+y | Z_n=x) = \omega_x ,
\end{align}  
for all $n\geq 0$, and with $Z_0=0$. Note that then expectations such as $\EE_p[Z_m]$ are actually independent of $p$.  

Besides a law of large numbers, the regeneration structure also implies a central limit theorem 
\begin{align}
\frac{1}{\sqrt{n}}(X_n-nv(p)) \xrightarrow[n\to\infty]{d} \mathcal N(0,\sigma^2(p)) ,  
\end{align}
where the limit is a normal distribution with variance
\begin{align}
\sigma^2(p) =  p \EE_p[(X_{\tau_1}-\tau_1v(p))^2] .
\end{align}
In the following, we study $v(p)$ and $\sigma^2(p)$ as a function of $p$.

\section{Main results}
\label{sec:results}

Our first result concerns the behavior of the speed $v(p)$ as a function of the renewal parameter $p$.

\begin{thm}\label{thm_mainspeed} 
For the effective speed $p\mapsto v(p)$ the following holds. 
\begin{enumerate}
\item $v$ is continuous on $[0,1]$. 
\item $v$ is infinitely often differentiable on $(0,1]$, with $k$-th derivative given by 
\begin{align}\label{vAbleitungen}
v^{(k)}(p) = (-1)^k\frac{k!}{p^{k-1}}\left(
\EE_p[Z_{N_{k-1}}]
- 2 \EE_p[Z_{N_{k}}]
+ \EE_p[Z_{N_{k+1}}]
\right),
\end{align}
where under $\PP_p$, the random variable $N_k$, $k\geq 1,$ follows a negative binomial distribution with parameters $k$ and $p$, independent of the walk and the environment, with $N_0=0$.

\item If $\EE_p[Z_1]>0$ and $  \Var(\omega_0)>0$, the speed has a strict local maximum in $p=1$.

\end{enumerate}
\end{thm}

We remark that analyticity on $(0,1]$ is a direct consequence of a power series representation of $v(p)$, see (\ref{vpowerseries}). 
Note that the process $(Z_n)_{n\geq 0}$ appearing in Equation \eqref{vAbleitungen} is moving in the first environment only, but the distribution of the random variables $N_k$ depends on $p$. 
Finally, we remark that the third statement in Theorem \ref{thm_mainspeed} implies that if $\EE_p[Z_1]<0$ and $  \Var(\omega_0)>0$, the speed has a strict local minimum in $p=1$, which follows by considering the reflected environment $\hat\omega$ with $\hat{\omega}^{(k)}_x=1-\omega^{(k)}_{-x}$ for $x\in\ZZ$ and $k\geq 0$.

It is an interesting question to ask for the monotonicity of $p\mapsto v(p)$. A characteristic feature of random walks in random environments is the strong influence of \emph{traps} on the process. At large times, a typical position of the particle is in a trap, a local configuration of the environment, which holds the process for a long time. In a one-dimensional environment, the concept of such traps can be made precise with the help of potential valleys, see \cite{sinai1983limiting,golosov1984localization,gantert2010infinite}. In our model, the regeneration of the environment will free the process of the current trap, so it is natural to expect $v(p)$ to be increasing in $p$, when the mean drift $\EE_p[Z_1]$ is positive. Theorem \ref{thm_mainspeed} shows that this is indeed true near $p=1$. However, we obtain as a consequence of the following proposition that the speed is, in general, \emph{not} increasing in $p$.   

In order to formulate the statement, we define
\begin{align*}
    \rho_x \coloneqq \frac{1-\omega_x}{\omega_x},\ x\in\ZZ.
\end{align*}
We then require 
\begin{align}\label{con1}
    \mathrm{E}_p[|\log \rho_0|]<\infty,\qquad
\end{align}
and the existence of an $\varepsilon>0$ such that
\begin{align}\label{con2}
\mathrm{E}_p[\rho_0^\varepsilon]<\infty.
\end{align}
These are conditions on the environment law only, and do not depend on $p$. 

\begin{proposition}\label{prop:mainspeed2} 
Suppose that \eqref{con1} and \eqref{con2} hold and
\begin{align}\label{con0}
\mathrm{E}_p[\log \rho_0]<0,\qquad
\mathrm{E}_p[\rho_0]\ge 1.
\end{align}
Then the right-sided derivative of the speed at $p=0$ is infinite:
\[v'(0)=\infty.\]
\end{proposition}

We recall that by the classical work of \cite{solomon1975random}, the condition $\mathrm{E}_P[\log \rho_0]<0$ implies that the random walk in a static environment is transient to the right, and $\mathrm{E}_P[\rho_0]\ge 1$ corresponds to the zero-speed regime, such that $v(0)=0$. 
From Proposition \ref{prop:mainspeed2}, we conclude that the speed $v(p)$ is, in general, not monotone.

\begin{cor}\label{cor1}
There are environment laws, such that $v$ is not monotone: 
Suppose that  \eqref{con1}, \eqref{con2} and \eqref{con0} hold, and that $\EE_p[Z_1]<0$. Then $v(0)=0$, $v'(0)=\infty$ and $v(1)=\EE_p[Z_1]<0$. Consequently, $v$ is not monotone.  
\end{cor}

We remark, that Corollary \ref{cor1} holds analogously when we consider an environment law such that the random walk in static random environment is transient to the left with zero speed and $\EE_p[Z_1] > 0$.

It is easy to give explicit distributions of transition probabilities for which these conditions are satisfied, and we consider as an example an environment law defined by 
\begin{align} \label{eq:example}
    \mathrm P_p\left(\omega_0 = \frac{1}{10}\right) = \alpha,\quad \mathrm P_p\left(\omega_0 = \frac{3}{5}\right) = 1-\alpha
\end{align}
for an $\alpha \in (0,1)$.
We obtain that $\EE_p[Z_1]>0$ if $\alpha < \frac{1}{5}$. In addition the random walk in a static environment is transient to the left if $\alpha > \frac{\log(3/2)}{\log(27/2)} \approx 0,156$. 

This gives us that, e.g., for $\alpha = 0.19$, the speed $v(p)$ is not monotone. This behavior is also observed in simulations of the speed for this particular environment law; see Figure \ref{fig:v(p)}.
On the other hand, for $\alpha<\frac{1}{25}$ we have that the random walk in static environment is transient to the right with $v(0)>0$. In general, we expect that $v(p)$ is monotone in this situation. A simulation for $\alpha=0.03$ supports this conjecture.  
Furthermore, we see, as stated in Theorem \ref{thm_mainspeed}.3, that $v(p)$ has a local maximum in $p=1$. 

\begin{figure}[H]
\centering

\begin{minipage}{0.48\textwidth}
    \centering
    \includegraphics[width=\textwidth, height= 6 cm]{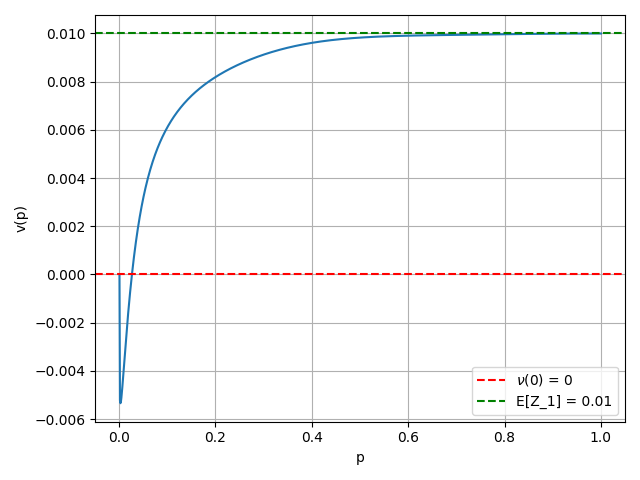}

\end{minipage}
\hfill
\begin{minipage}{0.48\textwidth}
    \centering
    \includegraphics[width=\textwidth, height= 6cm]{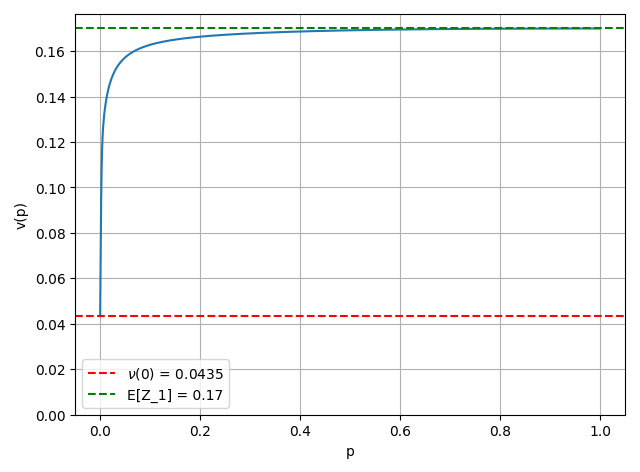}
\end{minipage}

\caption{Simulation of the speed $v(p)$ for the environmental distribution given in \eqref{eq:example} for $\alpha=0.19$ (left) and $\alpha=0.03$ (right).}
\label{fig:v(p)}

\end{figure}


The next statement concerns the behavior of the diffusivity $\sigma^2(p)$. 

\begin{thm}\label{thm:mainvar}
For the effective diffusivity $p\mapsto \sigma^2(p)$ we have the following: 
\begin{enumerate}
\item $\sigma^2$ is infinitely often differentiable on $(0,1]$.
\item If $\EE_p[Z_1]\neq 0$ and $  \Var(\omega_0)>0$, the diffusivity has a strict local minimum in $p=1$. 

\end{enumerate}
\end{thm}

In a static environment, the random walk is strongly localized in traps in the environment, which may be formalized as valleys in the potential landscape \cite{sinai1983limiting,golosov1984localization}. This trapping typically leads to a small variance, in Sinai's regime of order $\log(n)^2$. Regenerating the environment frees the random walk from such traps allowing for larger fluctuations. On the other hand, fluctuations of the location of potential valleys also contribute to the asymptotic variance, an effect that is lost in particular when $p$ is closer to 1. This can explain the local minimum of the variance at $p=1$, although it is hard to quantify this effect more precisely. As can be seen in simulations, $\sigma^2$ exhibits only a very slight decrease, see Figure \ref{fig:diffusivity1}, where we consider the environment law defined by 
\begin{align}\label{eq:example2}
\mathrm P_p \left( \omega_0 = \frac{1}{5}\right) = \frac{3}{8},\quad \mathrm P_p\left(\omega_0  = \frac{4}{5}\right) = \frac{5}{8}.
\end{align}

\begin{figure}[H]
\centering

\begin{minipage}{0.48\textwidth}
    \centering
    \includegraphics[width=\textwidth]{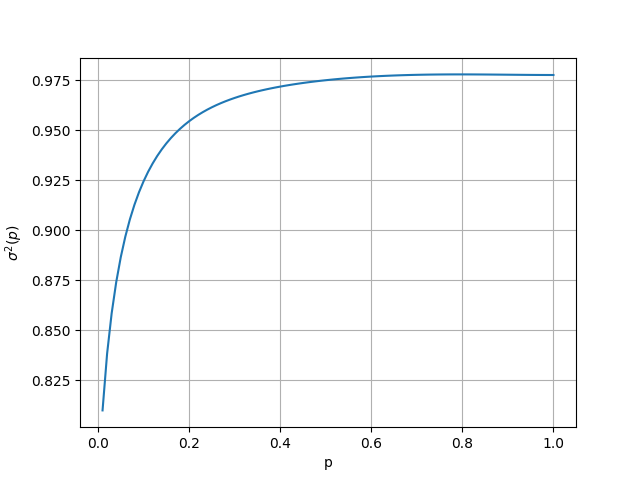}
\end{minipage}
\hfill
\begin{minipage}{0.48\textwidth}
    \centering
    \includegraphics[width=\textwidth]{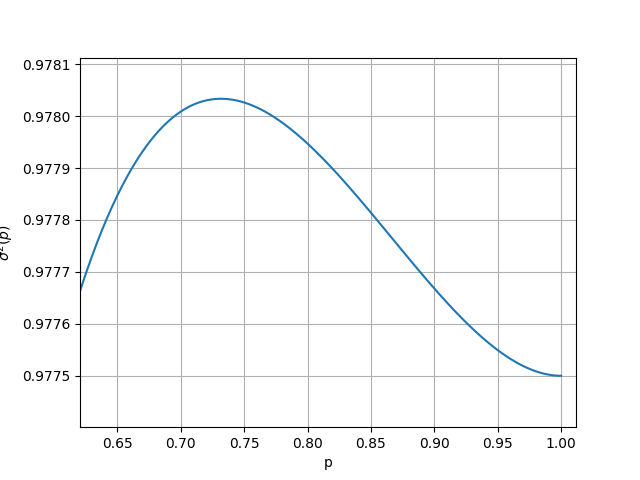}
\end{minipage}
\caption{Simulation of the diffusivity $\sigma^2(p)$ for the environmental distribution given in \eqref{eq:example2}. The right panel shows a magnified view for values of $p$ close to $1$.}
\label{fig:diffusivity1}

\end{figure}

The limit of the diffusivity at $p=0$ is more involved than continuity of the speed. Indeed, in a static random environment, the correct rescaling of $(Z_n)_{n\geq 0}$ to obtain a nondegenerate limit distribution can vary between $1/\log(n)^2$ in Sinai's regime and $1/n^\alpha$ for stable regimes. It is therefore not meaningful to compare $\sigma^2(p)$ for $p>0$ with the asymptotic variance in a static environment.

\section{Proofs}
\label{sec:proofs}

\subsection{Proof of Theorem \ref{thm_mainspeed} 1., continuity of the speed}
We can find a power series representation of the speed, which implies the continuity of $v$ on $(0,1]$.
We have:
\begin{align}
\label{vpowerseries}
v(p)=
p\EE_p[X_{\tau_1}]
&= p\sum_{m=1}^\infty \EE_p[X_{\tau_1}\mid \tau_1=m]\cdot \PP_p(\tau_1=m) \\ \nonumber
&= p^2\sum_{m=1}^\infty \EE_p[Z_m](1-p)^{m-1},
\end{align}
where the last identity applies because $\tau_1$ is geometrically distributed.

The power series has at least a radius of convergence $R=1$, so the right-hand side of (\ref{vpowerseries}) defines a function that is continuous on $(0,2)$ and agrees with $v$ on $(0,1]$. Hence $v$ is continuous on $(0,1]$. 

For the proof of continuity in $p=0$, we consider for $p\in (0,1]$:
\begin{align*}
v(p) 
= p\EE_p[X_{\tau_1}]
= \EE_p\left[\frac{X_{\tau_1}}{\tau_1}p\tau_1\right].
\end{align*}

Under $\PP_p$, the random variable $\tau_1$ is geometrically distributed with parameter $p$. There exists a coupling of random variables $\tau_1^p$, $p\in (0,1]$, such that $\tau_1^p$ is geometrically distributed with parameter $p$ and $\tau_1^p\leq \tau_1^{p'}$ for $p'\leq p$. We may assume that these random variables are defined on the same probability space as the sequence of environments and the process and denote the joint annealed law by $\hat{\PP}$. 
We then have $\hat \PP(\tau_1^p\leq x)\to 0$ as $p\to 0$ for all $x\in\mathbb R$. Thus, ${\tau}_1^p\xrightarrow[p\to 0]{}\infty$ almost surely. We therefore have $\hat \PP$-almost surely 
\begin{align*}
    \lim_{p\to 0}\frac{X_{\hat{\tau}_1^p}}{\hat{\tau}_1^p} 
    = \lim_{p\to 0}\frac{Z_{\hat{\tau}_1^p}}{\hat{\tau}_1^p} 
    = \lim_{n \to\infty}\frac{Z_{n}}{n} 
    = v(0) .
\end{align*}
Furthermore, we have $p\tau_1^p \xrightarrow[p\to 0]{d} \hat\tau\sim \text{Exp}(1)$ in distribution under $\hat \PP$, which follows from 
\begin{align*}
\hat{\EE}\left[e^{it(p\tau_1^p)}\right] 
 = \frac{pe^{itp}}{1-(1-p)e^{itp}}\xrightarrow[p\to 0]{}  \frac{1}{1-it} 
= \hat{\EE}[e^{it \hat\tau}] .
\end{align*}
Moreover, the family $(p\tau_1^p)_{p\in (0,1]}$ is bounded in $L^2(\hat{\PP})$ and is in particular uniformly integrable. Since $|X_n/n|\leq 1$, the same is true for $(X_{\tau_1^p}/\tau_1^p\cdot p\tau_1^p)_{p\in (0,1]}$. This yields 
\begin{align*}
\lim_{p\to 0}v(p)
= \lim_{p\to 0}\EE_p\left[\frac{X_{\tau_1}}{\tau_1}p\tau_1\right]
= \lim_{p\to 0}\hat{\EE}\left[\frac{X_{\tau_1^p}}{\tau_1^p}p\tau_1\right]
= \hat{\EE}[v(0)\hat\tau ] = v(0) . 
\end{align*}


\subsection{Proof of Theorem \ref{thm_mainspeed} 2., differentiability of the speed}

Recall the power series representation of the speed in \eqref{vpowerseries}. The right-hand side of (\ref{vpowerseries}) is infinitely differentiable on $(0,2)$ and since the right-hand side of (\ref{vpowerseries}) agrees with $v(p)$ for $p\in(0,1]$, it follows that $v$ is infinitely differentiable on (0,1]. 

Recall that under $\PP_p$, the random variable $N_k$ is negatively binomial distributed with parameters $k$ and $p$. Conditioning on the value of $N_k$, we have
\begin{align}\label{3g}
    \EE_p\left[Z_{N_k}\right] 
    &= \sum_{m=1}^\infty \EE_p[Z_{N_k} \mid N_k = m]\cdot \PP_p({N_k = m})
    = \sum_{m=1}^\infty \EE_p[Z_{m}]\cdot \PP_p({N_k = m})\nonumber \\ 
    &= \sum_{m=k}^\infty \EE_p[Z_{m}]\cdot \binom{m-1}{k-1}p^k(1-p)^{m-k}, 
\end{align}
which gives again a power series representation convergent on $(0,2)$, since $|\EE_p[Z_{m}]|\leq m$. 

We prove the formula (\ref{vAbleitungen}) by induction. 
Using formula \eqref{vpowerseries}, the first derivative of $v$ is given by 
\begin{align*}
    v'(p) & = \sum_{m=1}^\infty \EE_p[Z_m]\left(2p(1-p)^{m-1}-p^2(m-1)(1-p)^{m-2}\right) \\
    & =  2 \EE_p[Z_{N_1}] - \EE_p[Z_{N_2}],
\end{align*}
which agrees with the right hand side of \eqref{vAbleitungen} for $k=1$. 

For $k\geq 1$, we may compute inductively 
\begin{align} \label{eq:vdiffrecursion}
    v^{(k+1)}(p) \nonumber 
    &= \frac{d}{dp} \left((-1)^k\frac{k!}{p^{k-1}}\left(
    \EE_p[Z_{N_{k-1}}]
    - 2 \EE_p[Z_{N_k}]
    + \EE_p[Z_{N_{k+1}}]
    \right)\right) \\ \nonumber
    &= (-1)^{k+1}\frac{k!}{p^k}\Bigg((k-1)\left(\EE_p[Z_{N_{k-1}}]
    - 2 \EE_p[Z_{N_k}]
    + \EE_p[Z_{N_{k+1}}]\right) \\ 
    &\qquad - p\left(\frac{d}{dp}\left(\EE_p[Z_{N_{k-1}}]
    - 2 \EE_p[Z_{N_k}]
    + \EE_p[Z_{N_{k+1}}]\right)\right)\Bigg).
\end{align}
We claim that
\begin{align}\label{5a}
    \frac{d}{dp}\EE_p\left[Z_{N_k}\right] = \frac{k}{p}\left(\EE_p\left[Z_{N_k}\right]-\EE_p\left[Z_{N_{k+1}}\right]\right)
\end{align}
for all $k\geq 1$.
Indeed, differentiating \eqref{3g} and using $(m-k)\binom{m-1}{k-1}= k \binom{m-1}{k}$, we obtain 
\begin{align*}
    &\frac{d}{dp}\EE_p[Z_{N_k}] \\
    & = \sum_{m=k}^\infty \EE_p[Z_{m}]\cdot \binom{m-1}{k-1}kp^{k-1}(1-p)^{m-k} - \sum_{m=k+1}^\infty \EE_p[Z_{m}]\cdot \binom{m-1}{k-1} p^k(m-k)(1-p)^{m-k-1} \\ \nonumber
    &= \frac{k}{p}\left(\sum_{m=k}^\infty \EE_p[Z_{m}]\cdot \binom{m-1}{k-1} p^{k}(1-p)^{m-k} - \sum_{m=k+1}^\infty \EE_p[Z_{m}] \binom{m-1}{k} p^{k+1}(1-p)^{m-(k+1)}\right) \\ \nonumber
    &= \frac{k}{p}\left(\EE_p[Z_{N_k}]-\EE_p[Z_{N_{k+1}}]\right),
\end{align*}
which proves (\ref{5a}). Returning to \eqref{eq:vdiffrecursion}, Identity (\ref{5a}) implies
\begin{align*}
    &(k-1)(\EE_p[Z_{N_{k-1}}]
    - 2 \EE_p[Z_{N_k}]
    + \EE_p[Z_{N_{k+1}}]) 
    - p\left(\frac{d}{dp}(\EE_p[Z_{N_{k-1}}]
    - 2 \EE_p[Z_{N_k}]
    + \EE_p[Z_{N_{k+1}}])\right) \nonumber \\ \nonumber
    &= p \frac{d}{dp}\EE_p[Z_{N_{k-1}}]-p\frac{k-1}{k}\frac{d}{dp}\EE_p[Z_{N_k}]
    - p\left(\frac{d}{dp}(\EE_p[Z_{N_{k-1}}]
    - 2 \EE_p[Z_{N_k}]
    + \EE_p[Z_{N_{k+1}}])\right) \\ 
    &= \frac{k+1}{k}p\frac{d}{dp}\EE_p[Z_{N_k}] -p \frac{d}{dp}\EE_p[Z_{N_{k+1}}].
\end{align*}
Thus, we have
\begin{align*}
    v^{(k+1)}(p)
    &= (-1)^{k+1}\frac{k!}{p^k}\left(\frac{k+1}{k}p\frac{d}{dp}\EE_p[Z_{N_k}] -p \frac{d}{dp}\EE_p[Z_{N_{k+1}}]\right) \\ \nonumber
    &= (-1)^{k+1}\frac{k!}{p^k}\left((k+1)(\EE_p[Z_{N_k}]-\EE_p[Z_{N_{k+1}}]) - (k+1)(\EE_p[Z_{N_{k+1}}]-\EE_p[Z_{N_{k+2}}])\right) \\ \nonumber
    &= (-1)^{k+1}\frac{(k+1)!}{p^k}\left(\EE_p[Z_{N_k}]-2\EE_p[Z_{N_{k+1}}] + \EE_p[Z_{N_{k+2}}]\right),
\end{align*}
which proves the claim by induction. 

\subsection{Proof of Theorem \ref{thm_mainspeed}. 3., the derivatives at $p=1$}

From (\ref{vAbleitungen}) we have 
\begin{align} \label{eq:versteableitung}
    v'(1) = -(-2\EE_1[Z_{N_1}]+\EE_1[Z_{N_2}])
    = 2\EE_1[Z_1]-\EE_1[Z_2] , 
\end{align}
since $N_k =k$ almost surely under $\PP_1$. We have $\EE_p[Z_2]=2\EE_p[Z_1]$ and consequently $v'(1)=0$, since $\EE_p[Z_2]=\EE_p[Z_1] + \EE_p[Z_2-Z_1]$ and $Z_1$ and $Z_2-Z_1$ have the same distribution under the annealed law $\PP_p$. 
We remark that $\EE_p[Z_k]=k\EE_p[Z_1]$ does not hold for $k\geq 3$, since the random walk is not a Markov chain under the annealed law and the increments are, in general, not identically distributed. 

The lemma below yields $v''(1)<0$ if $\EE_p[Z_1]>0$ and $\Var(\omega_0)>0$. Hence, the speed attains a strict local maximum at $p=1$, which proves part 3 of Theorem \ref{thm_mainspeed}.

\begin{lem}\label{lem:speedat1}
The second derivative of $v$ at $p=1$ is given by
\begin{align*}
v''(1) =
-4\mathrm{Var}\left(\omega_0\right)\EE_1[Z_1].
\end{align*}
\end{lem}
\begin{proof}
According to Theorem 2.1.2, $v$ is infinitely differentiable on $(0,1]$ and the second derivative of $v$ at $p=1$ equals

\begin{align}\label{lemma3.1C}
    v''(1) 
    &= (-1)^2 2! (\EE_1[Z_{N_1}]-2\EE_1[Z_{N_2}] +\EE_1[Z_{N_3}]) \\ \nonumber
    &= 2(\EE_1[Z_1]-2\EE_1[Z_2]+\EE_1[Z_3])
    = 2(\EE_1[Z_3]-3\EE_1[Z_1]) .
\end{align}
We have
\begin{align*}
    \EE_1[Z_3]
    =3 \PP_1(Z_3=3) + \PP_1(Z_3=1) - \PP_1(Z_3=-1) - 3 \PP_1(Z_3=-3)
\end{align*}
and proceed by computing these probabilities as expectation with respect to the environment law $\mathrm{P}_p$. Since we only consider random variables with distribution not depending on $p$, we omit the subscript. 
Additionally, we write $\Bar{\omega}$ for $\mathrm{E}[\omega_0]$.
Since the random variables $\omega_x$, $x\in\mathbb Z$, are i.i.d., we obtain
\begin{align*}
    \PP(Z_3=3) = \mathrm{E}\left[\omega_0\omega_1\omega_2\right]
    =  \mathrm{E}\left[\omega_0\right]^3 = \Bar{\omega}^3,
\end{align*}
and 
\begin{align*}
    \PP(Z_3=1) & = \mathrm{E}[\omega_0 \omega_1 (1-\omega_2)] + \mathrm{E}\left[\omega_{0}(1-\omega_{1})\omega_{0}\right]
    + \mathrm{E}\left[(1-\omega_0)\omega_{-1}\omega_0\right] \\ \nonumber
    &= \bo^2(1-\Bar{\omega})+ (1-\Bar{\omega})\mathrm{E}[\omega_0^2]+\Bar{\omega}\mathrm{E}[(1-\omega_0)\omega_0].
\end{align*}
The calculation of the other probabilities is carried out in the same way and yields
\begin{align*}
    \PP(Z_3=-1)
    &= \bo(1-\bo)^2+\bo \mathrm{E}[(1-\omega_0)^2]+(1-\bo)\mathrm{E}[\omega_0(1-\omega_0)], \\
    \PP(Z_3=-3)&=(1-\bo)^3.
\end{align*}

It follows, that
\begin{align*}
    \EE[Z_3] =  &3\bo^3
    + \bo^2(1-\Bar{\omega})+ (1-\Bar{\omega})\mathrm{E}[\omega_0^2]+\Bar{\omega}\mathrm{E}[(1-\omega_0)\omega_0] \\ \nonumber
    &- \bo(1-\bo)^2-\bo \mathrm{E}[(1-\omega_0)^2]-(1-\bo)\mathrm{E}[\omega_0(1-\omega_0)]
    - 3(1-\bo)^3.
\end{align*}
Furthermore, we may write
\begin{align*} 
    3\EE[Z_1] 
    &= 3(2\Bar{\omega}-1)
    = 3\left(\bo^3 + \bo^2(1-\Bar{\omega})- \bo(1-\bo)^2 - (1-\bo)^3\right) \\ \nonumber
    &= 3\bo^3
    + 3\bo^2(1-\Bar{\omega})-3
    \bo(1-\bo)^2 - 3(1-\bo)^3,     
\end{align*}
which gives
\begin{align}\label{Lemma3.1E}
    \EE[Z_3]-3\EE[Z_1]
    &= (1-\bo)(\mathrm{E}[\omega_0^2]-\bo^2)
    + \bo (\mathrm{E}[(1-\omega_0)\omega_0]-(1-\bo)\bo)\\ \nonumber
    &\, - \bo(\mathrm{E}[(1-\omega_0)^2]-(1-\bo)^2)
    - (1-\bo)(\mathrm{E}[\omega_0(1-\omega_0)]-\bo(1-\bo)) \\ \nonumber
    &= (1-\bo)\mathrm{Var}(\omega_0)
    - \bo \mathrm{Var}(1-\omega_0) + (2\bo -1)(\mathrm{E}[\omega_0(1-\omega_0)]-\bo(1-\bo)) \\ \nonumber
    &= (1-\bo)\mathrm{Var}(\omega_0)- \bo \mathrm{Var}(\omega_0) - (2\bo-1)\mathrm{Var}(\omega_0) \\ \nonumber
    &= 2(1-2\bo)\mathrm{Var}(\omega_0)
    = -2 \EE[Z_1]\mathrm{Var}(\omega_0).
\end{align}
Substituting into (\ref{lemma3.1C}) yields 
\begin{align*}
    v''(1)
    = 2(\EE[Z_3]-3\EE[Z_1])
    = -4\mathrm{Var}(\omega_0)\EE[Z_1].
\end{align*}
\end{proof}

\subsection{Proof of Proposition \ref{prop:mainspeed2}, $v'(0) = \infty$}

When $\mathrm{E}_p[\log \rho_0]<0$, but $\mathrm{E}_p[\rho_0]\geq 1$, the random walk in static random environment is transient to the right, but with zero speed, $v(0)=0$. The derivative of the speed at $p=0$ is then given by 
\begin{align}\label{4a}
    v'(0) = \lim_{p\to 0}\frac{v(p)-v(0)}{p}
    = \lim_{p\to 0} \frac{v(p)}{p}
    = \lim_{p\to 0}\EE_p[X_{\tau_1}].
\end{align}
The main part of this proof will be to show that the expectation of the static random walk $\EE_p[Z_m]$ goes to infinity if $m\to\infty$, see Lemma \ref{lem:erwnachrechts} below. 
We remark that if the random walk in static random environment is transient to the right with positive speed, the convergence of $\EE_p[Z_m]$ to $\infty$ follows directly from dominated convergence, as $v(0)=\lim_{m\to\infty}\frac{1}{m}\EE_p[Z_m]>0$.

Since $\PP_p(\tau_1\leq m)\to 0$ as $p\to0$, the statement of  Lemma \ref{lem:erwnachrechts} implies
\begin{align*}
    \EE_p[X_{\tau_1}] 
    =\sum_{m=1}^\infty \EE_p[Z_m]\PP_p(\tau_1=m) \xrightarrow[p\to 0]{} \infty ,
\end{align*}
and thus the proposition follows.

\begin{lem} \label{lem:erwnachrechts}
If the distribution of the environment satisfies conditions (\ref{con1}), (\ref{con2}) and (\ref{con0}), we have
$\lim_{m\to\infty}\EE_p[Z_m]=\infty$. 
\end{lem}

\begin{proof}
Similar to the proof of Theorem \ref{thm_mainspeed}, we omit the subscript $p$ in the proof, since the distribution of $Z_m$ is independent of the renewal parameter. 
As mentioned above, the RWRE $(Z_n)_{n\geq 0}$ is transient to the right. Almost surely, this implies the existence of a sequence $(R_k)_{k\in\mathbb N}$ of \textit{regeneration times} (or \textit{renewal times}), satisfying for every $k\in\mathbb N$
\begin{align*}
Z_l<Z_{R_k}\quad\text{for all }l<R_k,
\qquad
Z_m\ge Z_{R_k}\quad\text{for all }m\ge R_k.
\end{align*}
Thus, at time $R_k$, the random walk reaches a new point $Z_{R_k}$ and then never visits points to the left of $Z_{R_k}$ again. 
The formal definition is
\begin{align*}
    R_1 \coloneqq \inf\{n> 0 \mid \max_{l<n}Z_l < Z_n \leq \min_{m\geq n}Z_m \}
\end{align*}
and for $k\geq 2$
\begin{align*}
    R_k \coloneqq \inf \{n>R_{k-1} \mid \max_{l<n} Z_l < Z_n \leq \min_{m\geq n}Z_m\}.
\end{align*}
For more details, especially for a proof of the existence of these regeneration times we refer to \cite{sznitman1999law}, who considered the renewal structure in a multidimensional context. 
Based on proofs in \cite{guo2019berry,sznitman2004topics}, \cite{avena2025gaussian} (see Lemma B.3 therein) have formulated that under our assumptions on the distribution of the environment there exists a constant $c>0$ such that 
\begin{align}\label{eq:regtimeexpmoment}
\EE[e^{cZ_{R_1}}]<\infty .
\end{align}
By construction, $Z_{R_n}\geq n$, so that for all $m\in\mathbb N$ and $n\in \mathbb N$, 
\begin{align}\label{4b}
    \EE[Z_m] \nonumber 
    &= \EE[Z_m\mathbbm{1}_{\{m< R_n\}}]+\EE[Z_m\mathbbm{1}_{\{m\geq R_n\}}] \\ \nonumber
    &\geq \EE\left[\min_{0\leq k \leq R_1}Z_k\right] + \EE[Z_{R_n}\mathbbm{1}_{\{m\geq R_n\}}] \\ \nonumber
    &\geq \EE\left[\min_{0\leq k \leq R_1}Z_k\right] + \EE[n\mathbbm{1}_{\{m\geq R_n\}}] \\ 
    &=\EE\left[\min_{0\leq k \leq R_1}Z_k\right] + n\PP(m\geq R_n).
\end{align}
We show that the first summand in (\ref{4b}) is finite. For this purpose, we define random variables $t_k, k\in\mathbb N_0$ which measure the maximal downward displacement of the random walk between the first visit of $k$ and the first visit of $k+1$. Formally, we define 
\begin{align*}
    T_k &\coloneqq \inf\{n\geq 0 \mid Z_n = k\}, k\in \mathbb Z, \\
    t_0 &\coloneqq - \min \{Z_n \mid 0 \leq n\leq T_1\}, \\
    t_k &\coloneqq k-\min\{Z_n \mid T_{k}\leq n\leq T_{k+1}\}, k\in\mathbb N.
\end{align*}
The shift-invariance of the environment distribution and the strong Markov property imply that the sequence $(t_k)_{k\geq 0}$ is stationary under $\PP$. 
We obtain
\begin{align}\label{4c1}
    \EE\left[\min_{0\leq k \leq R_1}Z_k\right] 
    \geq \EE\left[\sum_{k=0}^{Z_{R_1}} -t_k\right]
    = -\sum_{k=0}^\infty \EE\left[t_k \mathbbm{1}_{\{k\leq Z_{R_1}\}}\right].
\end{align}
By the Cauchy-Schwarz inequality, we have
\begin{align}\label{4c}
    \sum_{k=0}^\infty \EE\left[t_k \mathbbm{1}_{\{k\leq Z_{R_1}\}}\right]
    \leq \sum_{k=0}^\infty \EE[t_k^2]^{\frac{1}{2}}\cdot \PP(k\leq Z_{R_1})^{\frac{1}{2}}
    =\EE[t_0^2]^{\frac{1}{2}}\cdot \sum_{k=0}^\infty\PP(k\leq Z_{R_1})^{\frac{1}{2}}.
\end{align}
The finite exponential moment \eqref{eq:regtimeexpmoment} implies that the sum in \eqref{4c} is bounded by some $C<\infty$, so that
\begin{align*}
\EE\left[\min_{0\leq k \leq R_1}Z_k\right] \geq - C \EE[t_0^2]^{\frac{1}{2}}. 
\end{align*}
To show that $\EE[t_0^2]<\infty$, we use explicit formulas for hitting probabilities under the quenched law $P_\omega$. 
For $l\geq 1$ the event $\{t_0 = l\}$ implies that the random walk hits $-l$ before $1$, so the hitting-time $T_{-l}$ has to be smaller than $T_1$. In any environment $\omega$ with $|\log \rho_x|<\infty$ for all $x\in\mathbb Z$, this probability is given by
\begin{align*}
    P_\omega(T_{-l}<T_1)=\frac{1}{1+\sum_{i=-l+1}^0\left(\prod_{j=i}^0\rho_j^{-1}\right)} \leq \prod_{j=-l+1}^0\rho_j,
\end{align*}
see \cite[Equation 2.1.4]{zeitouni2004RWRE}. Since the probability is also bounded by 1, we have for any $\alpha\in[0,1]$
\begin{align*}
    P_\omega(T_{-l}<T_1)
    \leq \min\left\{1, \prod_{j=-l+1}^0\rho_j\right\}
    \leq \min\left\{1, \left( \prod_{j=-l+1}^0\rho_j\right)^\alpha \right\} 
    \leq \prod_{j=-l+1}^0\rho_j^\alpha.
\end{align*}
Combining these considerations we arrive at
\begin{align*}
    \PP(t_0=l)
    \leq \PP(T_{-l}<T_1)
    \leq \mathrm{E}\left[\prod_{j=-l+1}^0\rho_j^\alpha\right]
    =E_P \left[\rho_0^\alpha\right]^l
\end{align*}
for all $\alpha\in[0,1]$. Differentiating with respect to $\alpha$ yields
\begin{align}\label{4d}
    \frac{d}{d\alpha} \mathrm{E} \left[\rho_0^\alpha\right]\Big|_{\alpha=0}
    = \mathrm{E}\left[\frac{d}{d\alpha}\rho_0^\alpha \Big|_{\alpha=0}\right]
    = \mathrm{E}\left[\rho_0^\alpha \log \rho_0\Big|_{\alpha=0}\right]
    = \mathrm{E}[\log \rho_0 ]<0 , 
\end{align}
where the first identity holds because
\begin{align*}
    \left|\frac{d}{d\alpha} \rho_0^\alpha \right| = |\rho_0^\alpha \log \rho_0|
    \leq |\log \rho_0|\mathbbm{1}_{\{\rho_0\leq 1\}} + \rho_0^\varepsilon \frac{\varepsilon}{2}  \mathbbm{1}_{\{\rho_0 > 1\}} 
\end{align*}
for $0\leq \alpha \leq \varepsilon/2$, and the upper bound is integrable due to the assumptions (\ref{con1}) and (\ref{con2}).

As $\mathrm{E}[\rho_0^0]=1$, the inequality \eqref{4d} implies that there exists $\alpha^*\in(0,1)$ with $E_P[\rho_0^{\alpha^*}]<1$. Choosing such $\alpha^*$, 
\begin{align*}
    \EE\left[t_0^2\right]
    = \sum_{l=1}^\infty l^2 \PP(t_0=l)
    \leq \sum_{l=1}^\infty l^2 E_P[\rho_0^{\alpha^*}]^l
    <\infty , 
\end{align*}
which completes the proof that 
\begin{align*}
\EE\left[\min_{0\leq k \leq R_1}Z_k\right] >-\infty ,
\end{align*}
so the first summand in (\ref{4b}) is finite.

Thus, we can conclude from \eqref{4b} that 
\begin{align}
    \liminf_{m\to\infty} \EE[Z_m]
    \geq \EE\left[\min_{0\leq k \leq R_1}Z_k\right] + \liminf_{m\to\infty} n\PP(m\geq R_n)
    = \EE\left[\min_{0\leq k \leq R_1}Z_k\right] + n
\end{align}
for all $n\in \NN$.
This implies $\EE_p[Z_m]\xrightarrow[m\to\infty]{}\infty$, which completes the proof.
\end{proof}

\subsection{Proof of Theorem \ref{thm:mainvar} 1., differentiability of the  diffusivity}

Using the same conditioning as for the speed, we derive the following power series representation of the diffusivity:
\begin{align}\label{sigmapowerseries}
    \sigma^2 (p) 
    &= p\EE_p[(X_{\tau_1}-\tau_1v(p))^2]
    = p \sum_{m=1}^\infty \EE_p[(X_{\tau_1}-\tau_1 v(p))^2 | \tau_1=m]\cdot \PP_p(\tau_1=m) \nonumber \\ 
    &= p^2 \sum_{m=1}^\infty \EE_p\left[(Z_{m}-m v(p))^2\right] (1-p)^{m-1}.
\end{align}
We have $(Z_m-mv(p))^2\leq 4m^2$, so the series in \eqref{sigmapowerseries} converges absolutely for $p\in (0,2)$. 
Since $v$ is infinitely differentiable on $(0,1]$ by Theorem \ref{thm_mainspeed}, \eqref{sigmapowerseries} is infinitely differentiable on $(0,1]$. 
Hence, $\sigma^2$ is continuous and infinitely differentiable on $(0,1]$.

\subsection{Proof of Theorem \ref{thm:mainvar} 2., the derivative at $p=1$}

The lemma below yields $(\sigma^2)'(1)=0$ and $(\sigma^2)''(1) > 0$ if $\EE[Z_1] \neq 0$ and $\Var(\omega_0)>0$, so the diffusivity has a strict local minimum at $p=1$. As before, we omit the subscript $p$ in any expectation that does not depend on $p$.

\begin{lem}
The first and second derivatives of $\sigma^2$ at $p=1$ are given by
\begin{align*}
    (\sigma^2)'(1) &= 0, \\ 
    (\sigma^2)''(1) &= 24\EE[Z_1]^2 \Var \left(\omega_0\right).
\end{align*}

\begin{proof}
We use the notation
\begin{align*}
    a_m(p) \coloneqq 
    \EE\left[(Z_{m}-m v(p))^2\right].
\end{align*}
The first and second derivatives of this expression are then
\begin{align*}
    a_m'(p) &= -2mv'(p)(\EE[Z_m]-mv(p)), \\ \nonumber
    a_m''(p) &= -2mv''(p)(\EE[Z_m]-mv(p))+2m^2v'(p)^2.
\end{align*}
We have $v(1) = \EE[Z_1]$, $\EE[Z_2]=2\EE[Z_1]$ and $v'(1) = 0$, see \eqref{eq:versteableitung}, so that
\begin{align*}
    a_1(1) &= \EE\left[(Z_{1}- v(1))^2\right] 
    = \Var(Z_1)
    , \\ \nonumber
    a_1'(1) &= -2v'(1)(\EE[Z_1]-v(1)) = 0    , \\ \nonumber
    a_1''(1) &= -2v''(1)(\EE[Z_1]-v(1))+2v'(1)^2 = 0, \\ \nonumber
    a_2(1) &= \EE\left[(Z_{2}- 2v(1))^2\right] = \EE\left[(Z_{2}- \EE[Z_2])^2\right]
    = \Var(Z_2), \\ \nonumber
    a_2'(1) &= -4v'(1)(\EE[Z_2]-2v(1))= 0, \\ \nonumber 
    a_3(1) &=     \EE\left[(Z_{3}-3 v(1))^2\right] = \EE\left[(Z_{3}-3 \EE[Z_1])^2\right].
\end{align*}
We obtain the derivatives of $\sigma^2$ by differentiating the summands of the power series in (\ref{sigmapowerseries}):
\begin{align*}
(\sigma^2)'(p) 
&= \sum_{m=1}^\infty (1-p)^{m-1} \left(2pa_m(p)+p^2 a_m'(p) \right) 
- \sum_{m=2}^\infty(m-1)(1-p)^{m-2}p^2a_m(p) 
\end{align*}
and for the second derivative
\begin{align*} 
(\sigma^2)''(p) 
&= \sum_{m=2}^\infty 
-(m-1)(1-p)^{m-2}\left(2pa_m(p)+p^2a_m'(p)\right) \\ \nonumber
&+\sum_{m=1}^\infty (1-p)^{m-1} \left(2a_m(p)+4pa_m'(p)+p^2a_m''(p)\right)  \\ \nonumber
&+\sum_{m=3}^\infty(m-1)(m-2)(1-p)^{m-3}p^2a_m(p) \\ \nonumber
&- \sum_{m=2}^\infty (m-1)(1-p)^{m-2}\left(2pa_m(p)+p^2a_m'(p)\right).
\end{align*}
Evaluating the first derivative of the diffusivity at $p=1$ yields
\begin{align*}
    (\sigma^2)'(1) 
    = 2a_1(1) + a_1'(1) - a_2(1) = 2\Var(Z_1) - \Var(Z_2) = 0,
\end{align*}
since $Z_1$ and $Z_2-Z_1$ are i.i.d. under the annealed law $\PP_p$. 
For the second derivative of the diffusivity at $p=1$ we obtain
\begin{align} \label{3G}
    (\sigma^2)''(1) 
    &= -2a_2(1) - a_2'(1) 
    + 2a_1(1) + 4a_1'(1) + a_1''(1)
    + 2a_3(1) 
    - 2a_2(1) - a_2'(1) \nonumber  \\ \nonumber
    &= 2a_1(1) + 4a'_1(1)  + a''_1(1) - 4 a_2(1)   - 2a'_2(1) + 2a_3(1) \\ \nonumber 
	& = 2\Var(Z_1)-4\Var(Z_2) +2\EE\left[(Z_{3}-3 \EE[Z_1])^2\right] \\ 
    &= -6\Var(Z_1) + 2\EE\left[(Z_{3}-3 \EE[Z_1])^2\right].
\end{align}
We continue with the last summand and obtain:
\begin{align}\label{3H}
    \EE\left[(Z_{3}-3 \EE[Z_1])^2\right]
    &= \EE\left[\left((Z_{3}-\EE[Z_3])+(\EE[Z_3]-3 \EE[Z_1])\right)^2\right] \nonumber \\ \nonumber
    &= \Var(Z_3) + 2 \EE[(Z_{3}-\EE[Z_3])(\EE[Z_3]-3 \EE[Z_1])] + (\EE[Z_3]-3 \EE[Z_1])^2 \\ 
    &= \Var(Z_3) + (\EE[Z_3]-3 \EE[Z_1])^2.
\end{align}
In order to comupte $\Var(Z_3)$, we write $Z_3=Z_2+\Delta_3$ with $\Delta_3 = Z_3-Z_2$. We have
\begin{align*}
    \Var(Z_3) 
    &= \Var(Z_2 + \Delta_3)
    = \Var(Z_2) + \Var(\Delta_3) + 2\text{Cov}(Z_2, \Delta_3).
\end{align*}
Next, we consider the expectation of $Z_2\Delta_3$. As $Z_2 \in \{-2,0,2\}$ and $\Delta_3\in \{-1,1\}$ $\PP_p$-a.s., we obtain:
\begin{align*}
    \EE[Z_2\Delta_3]
    =&\ 2(\bo^3 + (1-\bo)^3)
    - 2((1-\bo)^2\bo + \bo^2(1-\bo)) \\ \nonumber
    =&\ 2\bo^2(2\bo-1) + 2(1-\bo)^2(1-2\bo) \\ \nonumber
    =&\ (2\bo^2-2(1-\bo)^2)\EE[Z_1]
    = (4\bo-2)\EE[Z_1]
    = 2\EE[Z_1]^2.
\end{align*}
Furthermore, we have
\begin{align*}
    \EE[\Delta_3] 
    &= \EE[Z_3]-\EE[Z_2]
    = (\EE[Z_3] - 3\EE[Z_1])+\EE[Z_1] \\ \nonumber   &=-2\EE[Z_1]\mathrm{Var}(\omega_0)+\EE[Z_1]
    = (1-2\Var(\omega_0))\EE[Z_1],
\end{align*}
where we used (\ref{Lemma3.1E}) for the third step. 
We arrive at 
\begin{align*}
    \text{Cov}(Z_2, \Delta_3) 
    &= \EE[Z_2\Delta_3] - \EE[Z_2]\EE[\Delta_3] \\ \nonumber
    &= 2\EE[Z_1]^2 - 2\EE[Z_1]^2(1-2\Var(\omega_0))
    = 4\Var(\omega_0)\EE[Z_1]^2
\end{align*}
and
\begin{align*}
    \Var(\Delta_3) 
    &= \EE[\Delta_3^2]-\EE[\Delta_3]^2
    = 1 - ((1-2\Var(\omega_0))\EE[Z_1])^2 \\ \nonumber
    &= 1 - (1-4\Var(\omega_0)+4\Var(\omega_0)^2)\EE[Z_1]^2
    \\ \nonumber
    &= 1-\EE[Z_1]^2+4\EE[Z_1]^2(\Var(\omega_0)-\Var(\omega_0)^2) \\ \nonumber
    &= \Var(Z_1) +4\EE[Z_1]^2(\Var(\omega_0)-\Var(\omega_0)^2).
\end{align*}
The variance of $Z_3$ is then given by 
\begin{align}\label{3I}
    \Var(Z_3) 
    &= 3\Var(Z_1) +4\EE[Z_1]^2(\Var(\omega_0)-\Var(\omega_0)^2) + 8\Var(\omega_0)\EE[Z_1]^2\\ \nonumber
    &= 3\Var(Z_1) +4\EE[Z_1]^2(3\Var(\omega_0)-\Var(\omega_0)^2).
\end{align}
Furthermore, we have by (\ref{Lemma3.1E})
\begin{align}\label{3J}
    (\EE[Z_3]-3 \EE[Z_1])^2
    = (-2\EE[Z_1]\Var(\omega_0))^2
    = 4\EE[Z_1]^2\Var(\omega_0)^2.
\end{align}
Combining now (\ref{3G}), (\ref{3H}), (\ref{3I}) and (\ref{3J}), we obtain 
\begin{align*}
    (\sigma^2)''(1)  
    &= -6\Var(Z_1) + 2\EE\left[(Z_{3}-3 \EE[Z_1])^2\right] \\ \nonumber
    &=-6\Var(Z_1) + 2\Var(Z_3) + 2 \left(\EE[Z_3]-3 \EE[Z_1]\right)^2 \\ \nonumber
    &= 8 \EE[Z_1]^2\left(3\Var(\omega_0)-\Var(\omega_0)^2\right) + 8 \EE[Z_1]^2\Var(\omega_0)^2 \\ \nonumber
    &= 24\EE[Z_1]^2 \Var(\omega_0).
\end{align*}
\end{proof}
\end{lem}

\bibliographystyle{alpha}
\bibliography{bibRWDRE}

\bigskip

{\footnotesize
\noindent
TU Dortmund, \\
Fakult\"at f\"ur Mathematik, \\
Vogelpothsweg 87, 
44227 Dortmund, 
Germany, \\
raphael.martin@tu-dortmund.de\\
jan.nagel@tu-dortmund.de
}

\end{document}